\documentclass[a4paper,11pt]{article}
\usepackage{bbding,textcomp,amsfonts,amsthm,amsmath,mathrsfs}
\usepackage{amssymb}
\usepackage{multirow, url}
\usepackage[utf8x]{inputenc}
\usepackage{comment}
\usepackage[affil-it]{authblk}
\usepackage{titling}
\usepackage[T1]{fontenc}
\usepackage{ae,aecompl}
\usepackage{enumerate}

\newtheorem{theorem}{Theorem}

\newtheorem{lemma}[theorem]{Lemma}

\numberwithin{subcase}{case}

\newcommand{\ex}{\textrm{\upshape{ex}}}

\newcommand{\gal}{\textrm{\upshape{Gal}}}

\PrerenderUnicode{\unichar{355}}

\title{A note on projective norm graphs}

\author{Codru\unichar{355} Grosu\thanks{This research was supported by the Deutsche Forschungsgemeinschaft within the research training group `Methods for Discrete Structures' (GRK 1408).}}
\affil{\small{grosu.codrut@gmail.com}}
\date{}

\begin{document}

\maketitle

\abstract{The projective norm graphs $\mathcal{P}(q, 4)$ introduced by Alon, R\'onyai and Szab\'o are explicit examples of extremal graphs not containing $K_{4, 7}$. Ball and Pepe showed that $\mathcal{P}(q, 4)$ does not contain a copy of $K_{5, 5}$ either for $q \geq 7$, asymptotically improving the best lower bound for $\ex(n, K_{5, 5})$.

We show that these results can not be improved, in the sense that $\mathcal{P}(q, 4)$ contains a copy of $K_{4, 6}$ for infinitely many primes $q$.
}

\section{\normalsize Introduction}

Let $H$ be any graph. For any $n \geq 1$, the Tur\'an function $\ex(n, H)$ is defined as the maximum possible number of edges of an $H$-free graph on $n$ vertices. Although the precise value of the Tur\'an function is known in several cases, for arbitrary graphs $H$ the best result we have is an asymptotic estimate provided by the Erd\H os-Stone-Simonovits theorem:
\begin{equation}
\label{eq:ErdosStone}
\ex(n, H) = \left(1 - \frac{1}{\chi(H) - 1} + o(1)\right)\binom{n}{2},
\end{equation}
where $\chi(H)$ is the chromatic number of $H$. This estimate is particularly poor when $H$ is bipartite ($\chi(H) = 2$), as then it only tells us that $\ex(n, H) = o(n^2)$. Therefore it makes sense to try to obtain better estimates for bipartite graphs.

Even the case of complete bipartite graphs $K_{t, s}$ is not fully solved. K{\H o}v{\'a}ri, S\'os and Tur\'an  (\cite{Kovari54}) showed that $\ex(n, K_{t, s}) \leq c_{t, s} n^{2 - \frac{1}{t}}$ for any $s \geq t$, where $c_{t, s} > 0$ is a constant depending only on $t$ and $s$. While the upper bound is believed to be tight, the best general lower bound, obtained by the probabilistic method (\cite{Erdos74}), only gives
\begin{equation}
\label{eq:lowerBound}
\ex(n, K_{t, s}) \geq cn^{2-\frac{s+t-2}{st-1}}.
\end{equation}
Nevertheless, the upper bound is tight for $t=2$, as shown by Erd\H os, R\'enyi and S\'os (\cite{Erdos66}), and $t = 3$, as shown by Brown (\cite{Brown66}). On the other hand, the asymptotic behaviour of $\ex(n, K_{t, t})$ is not known for any $t \geq 4$.

A significant step towards a solution to this problem was made by Koll\'ar, R\'onyai and Szab\'o (\cite{Szabo95}), who constructed the norm graphs $\mathcal{G}(q, t)$, a family of graphs which are extremal for $K_{t, s}$, for any $s \geq t!+1$. This construction was further refined by Alon, R\'onyai and Szab\'o (\cite{Szabo99}), who introduced the projective norm graphs $\mathcal{P}(q, t)$, graphs which are extremal for $K_{t, s}$, for any $s \geq (t-1)!+1$.

The construction of $\mathcal{P}(q, t)$ is as follows. Let $t \geq 3$ and $q$ be a prime power. We let $\mathbb{F}_{q}$ be the finite field with $q$ elements.

We define $\mathcal{P}(q, t)$ as the graph on vertex set $\mathbb{F}_{q^{t-1}} \times \mathbb{F}_q^*$, with an edge between two vertices $(\alpha, a)$ and $(\beta, b)$ if and only if $N(\alpha+\beta) = ab$. Here $N(x)$ denotes the norm of the element $x$ relative to the extension $\mathbb{F}_{q^{t-1}}/\mathbb{F}_q$. Then the following holds.

\begin{theorem}[\cite{Szabo99}]
\label{thm:szabo}
Any $t$ distinct vertices in $\mathcal{P}(q, t)$ have at most $(t-1)!$ common neighbors.
\end{theorem}

Later several other families of graphs avoiding large bipartite subgraphs were constructed (\cite{Bukh13}, \cite{Bukh15}). However, none of these constructions matches (or improves) the bound $s \geq (t-1)!+1$, and in fact to this day the projective norm graphs remain the best construction we have: for no $s \leq (t-1)!$ and $t$ large enough do we know an example of a graph improving \eqref{eq:lowerBound}.

For $t=4$, $\mathcal{P}(q, 4)$ avoids $K_{4, 7}$. It was shown by Ball and Pepe (\cite{Ball12}) that for $q \geq 7$, $\mathcal{P}(q, 4)$ further avoids $K_{5, 5}$, thus improving the lower bound on $\ex(n, K_{5, 5})$ to $\frac{1}{2}(1+o(1))n^{7/4}$. As the best lower bound on $\ex(n, K_{4, 4})$ is
\begin{equation*}
\ex(n, K_{4, 4}) \geq \frac{1}{2}n^{5/3} + o(n^{5/3}),
\end{equation*}
following from Brown's optimal construction for $K_{3, 3}$, it would be desirable to know whether $\mathcal{P}(q, 4)$ further avoids any of the subgraphs $K_{4, 4}, K_{4, 5}$ or $K_{4, 6}$. The main result of this note is the following.

\begin{theorem}
\label{thm:main}
There exists an infinite sequence of primes of density $\frac{1}{9}$ such that for any prime $p$ in this sequence, $\mathcal{P}(p, 4)$ contains a copy of $K_{4, 6}$. 
\end{theorem}

It would be very interesting to show that $\mathcal{P}(q, t)$ contains a copy of $K_{t, (t-1)!}$ for any $t \geq 4$ and infinitely many $q$. I strongly believe that this general claim should have a beautiful simple proof; but so far I have not found any.

As an amend to this, I show the following.
\begin{theorem}
\label{thm:sec}
For any $t \geq 4$ and any $m \geq 1$, there exists an infinite sequence of primes of positive density such that for any prime $p$ in this sequence, $\mathcal{P}(p, t)$ contains a copy of $K_{t-1, m}$. 
\end{theorem}
Of course it follows directly from the  K{\H o}v{\'a}ri-S\'os-Tur\'an theorem that $\mathcal{P}(q, t)$ contains arbitrarily large $K_{t-1, m}$. The point of Theorem \ref{thm:sec} is that it provides an explicit construction of such subgraphs.

The proof of Theorem~\ref{thm:main} is of the following nature. A relatively well-behaved copy of $K_{4, 6}$ is found in $\mathcal{P}(7, 4)$ using a computer search. Then using the approach of \cite{VuWood} and \cite{Grosu14}, this copy is lifted to infinitely many primes $p$. However, the techniques in \cite{Grosu14} only apply if the initial copy is in a finite field of very large characteristic. It is rather remarkable that the same methods can be used if the starting prime is small (in our case, $7$). Furthermore, an additional theoretical complication is introduced by the use of the norm; this requires a more delicate handling of the Galois groups involved.

\section{\normalsize{A Galois extension of $\mathbb{Q}$}}

The plan to prove Theorem~\ref{thm:main} is the following. We shall first define two finite extensions $E /\mathbb{Q}, F / E$ such that $[F : E] = 3$. Then in $F/E \times E$ we will construct a copy of $K_{4, 6}$, with the norm condition fulfilled by $N_{F/E}(x)$. We will then show that $E$ maps in a suitable sense to $\mathbb{F}_p$, for infinitely many $p$, while $F$ maps to a degree $3$ extension of $\mathbb{F}_p$. As the norm essentially stays the same, this will give a copy of $K_{4, 6}$ in $\mathcal{P}(p, 4)$.

Let $\zeta$ be a third root of unity. Let $\theta_1$ be the root $\sqrt[3]{2}$ of the polynomial $x^3 - 2$ and $\theta_2$ the root $\sqrt[3]{3}$ of the polynomial $x^3 - 3$. Define $F := \mathbb{Q}(\zeta, \theta_2, \theta_1)$.

\begin{lemma}
The field $F$ is a Galois extension of $\mathbb{Q}$ of degree $[F : \mathbb{Q}] = 18$. A basis of $F/\mathbb{Q}$ is given by all elements of the form $\zeta^i\theta_1^j\theta_2^k$, with $0 \leq i \leq 1$ and $0 \leq j, k \leq 2$.
\end{lemma}
\begin{proof}
The roots of the polynomial $x^3 - 2$ are $\zeta^i\theta_1$ with $0 \leq i \leq 2$. Similarly, the roots of the polynomial $x^3 - 3$ are $\zeta^i\theta_2, 0 \leq i \leq 2$. Hence $F$ is the splitting field of $(x^3-2)(x^3 - 3)$, in particular a Galois extension of $\mathbb{Q}$.

Clearly $[\mathbb{Q}(\zeta) : \mathbb{Q}] = 2$ with basis $\{1, \zeta\}$, and $[F : \mathbb{Q}(\zeta)] \leq 9$. We show that the elements of $A := \{\theta_1^j\theta_2^k : 0 \leq j, k \leq 2\}$ are pairwise linearly independent over $\mathbb{Q}(\zeta)$. Indeed, if
\begin{equation*}
\alpha \theta_1^{j_1}\theta_2^{k_1} + \beta\theta_1^{j_2}\theta_2^{k_2} = 0, \quad \alpha, \beta \in \mathbb{Q}(\zeta)^*, (j_1, k_1) \neq (j_2, k_2),
\end{equation*}
then $ \theta_1^{j_1 - j_2}\theta_2^{k_1 - k_2} \in \mathbb{Q}(\zeta)$. Then $\mathbb{Q}(\zeta)$ contains a cube root of $2$, $3$, $6$ or $12$, and hence has degree at least $3$ over $\mathbb{Q}$, a contradiction.

Then by Theorem 1.3, \cite{carr09}, $A$ is linearly independent over $\mathbb{Q}(\zeta)$. Hence $F/\mathbb{Q}(\zeta)$ is an extension of degree $9$ and has $A$ as a basis. Then the elements $\zeta^i\theta_1^j\theta_2^k$ also form a basis of $F$ over $\mathbb{Q}$. 
\end{proof}

Let $S := \{\zeta^i\theta_1^j\theta_2^k :0 \leq i \leq 1, 0 \leq j, k \leq 2\}$ be the basis of $F /\mathbb{Q}$.

Let $G$ be the Galois group of $F/\mathbb{Q}$. Any element $\tau \in G$ is described by how it acts on $\zeta, \theta_1$ and $\theta_2$. Consider the maps
\begin{equation*}
\tau_1 : \zeta \rightarrow \zeta^2, \quad \tau_2 : \theta_1 \rightarrow \zeta\theta_1, \quad \tau_3 : \theta_2 \rightarrow \zeta\theta_2,
\end{equation*}
which act on one element of $\{\zeta, \theta_1, \theta_2\}$ and fix the remaining two. This defines $\tau_i$ at every element of $S$, and by linearity we can extend each $\tau_i$ to the whole of $F$. Then $\tau_i$ is multiplicative on $S$, in particular $\tau_i \in G$. It follows that $\tau_1, \tau_2$ and $\tau_3$ generate $G$ (as they generate $18$ distinct elements). But the subgroup $<\tau_2, \tau_3>$ is isomorphic to $\mathbb{Z}_3 \times \mathbb{Z}_3$, a direct product of two cyclic groups, and $\tau_1$ acts as conjugation. Thus $G$ is isomorphic to the semi-direct product $(\mathbb{Z}_3 \times \mathbb{Z}_3) \rtimes_\phi \mathbb{Z}_2$, with the map $\phi : \mathbb{Z}_3 \times \mathbb{Z}_3 \rightarrow \mathbb{Z}_3 \times \mathbb{Z}_3$ sending an element to its inverse.

Let $\sigma \in G$ be the Galois automorphism sending $\zeta \rightarrow \zeta$, $\theta_1 \rightarrow \zeta\theta_1$ and $\theta_2 \rightarrow \zeta^2\theta_2$. Then $\sigma$ corresponds to the element $(1, 2) \rtimes_\phi 0$, and so has conjugacy class $C_\sigma = \{\sigma, \sigma^2\}$.

Note that $\{1, \sigma, \sigma^2\}$ is a subgroup of $G$. Let $E \subset F$ be the fixed field of this subgroup. Then $\mathbb{Q}(\zeta) \subset E$. However, as $\sigma$ permutes the roots of $x^3-2$ cyclically, this polynomial is still irreducible over $E$. Then $F = E(\theta_1)$. The fields $E$ and $F$ are our desired extensions.

\section{\normalsize{The construction of $K_{4, 6}$}}

We are now ready to describe the construction of $K_{4, 6}$ in $F/E \times E$. We shall use the basis $1, \theta_1, \theta_1^2$ of $F/E$.

Consider the polynomial $g(x) := x^3 + 21x^2 + 3x+ 7$. It has the roots
\begin{equation*}
-4\cdot 6^{1/3} - 2\cdot6^{2/3} - 7, \quad -4\cdot 6^{1/3} \zeta - 2\cdot 6^{2/3} \bar{\zeta}- 7, \quad -4\cdot 6^{1/3} \bar{\zeta} - 2\cdot 6^{2/3} \zeta - 7,
\end{equation*}
which all lie in $F$. The most important thing for us is that $\sigma(6^{1/3}) = \sigma(\theta_1 \theta_2) = \zeta^3\theta_1 \theta_2 = 6^{1/3}$, and hence $\sigma$ fixes all the roots of $g$. Then all the roots of $g$ belong to $E$.

Now define $A := \{(0, 3), (1, 4), (2, 5), (\theta_1 + 1, 6)\}$. Further define 
\begin{equation*}
\begin{split}
B := &\left\{(\zeta^k \theta_1^2 - 1, 1) : 0 \leq k \leq 2\right\} \\
	&\quad \bigcup \left\{\left(-\frac{1-\eta}{4}\theta_1^2 - \frac{1+\eta}{2}\theta_1 - 1, \frac{1+3\eta^2}{4}\right) : \eta \textrm{ a root of $g$}\right\}.
\end{split}
\end{equation*}
Any element of $A \cup B$ is of the form $(\alpha, a)$ with $\alpha \in F$ and $a \in E$. The existence of a $K_{4, 6}$ is shown by the following.
\begin{lemma}
\label{lem:K46}
For any $(\alpha, u) \in A$ and any $(\beta, v) \in B$ we have $N_{F/E}(\alpha+\beta) = uv$.
\end{lemma}
\begin{proof}
Let $ x := a\theta_1^2 + b\theta_1 + c, a, b, c \in E$, be any element of $F$. The norm $N_{F/E}(x)$ is defined as the determinant of the linear map $y \rightarrow xy$. For the basis $1, \theta_1, \theta_1^2$ we get
\begin{equation*}
N_{F/E}(x) = \det\left(
\begin{array}{ccc}
c & 2a & 2b\\
b & c & 2a\\
a & b & c\\
\end{array}
\right) = c^3+2b^3+4a^3- 6abc.
\end{equation*}
By our choice of $B$, $\beta$ is of the form $a\theta_1^2 + b\theta_1 - 1$. Thus we need to check that the following holds:
\begin{align*}
N_{F/E}(a\theta_1^2 + b\theta_1 - 1) &= 4a^3 + 2b^3 + 6ab - 1 &= 3v,\\
N_{F/E}(a\theta_1^2 + b\theta_1) &= 4a^3 + 2b^3 &= 4v,\\
N_{F/E}(a\theta_1^2 + b\theta_1 + 1) &= 4a^3 + 2b^3 - 6ab + 1 &= 5v,\\
N_{F/E}(a\theta_1^2 + (b+1)\theta_1) &= 4a^3 + 2b^3 + 6b^2 + 6b + 2 &= 6v,\\
\end{align*}
for any $(\beta, v) \in B$. This readily follows if $(\beta, v)$ is of the form $(\zeta^k\theta_1^2-1, 1)$.

So assume $(\beta, v) = \left(-\frac{1-\eta}{4}\theta_1^2 - \frac{1+\eta}{2}\theta_1 - 1, \frac{1+3\eta^2}{4}\right)$ for some root $\eta$ of $g$.

As $\eta^3 = -21\eta^2-3\eta-7$, we have $(1-\eta)^3 = 8+24\eta^2$. Also $(1+\eta)^3 = -6-18\eta^2$. Consequently
\begin{equation*}
4\left(-\frac{1-\eta}{4}\right)^3 + 2\left(- \frac{1+\eta}{2}\right)^3 = -\frac{1}{16}(8+24\eta^2) + \frac{1}{4}(6 + 18\eta^2) = 1+3\eta^2 = 4v.
\end{equation*}
Furthermore,
\begin{equation*}
1 - 6\left(-\frac{1-\eta}{4}\right)\left(- \frac{1+\eta}{2}\right) = 1 - \frac{3}{4}(1 - \eta^2) = v.
\end{equation*}
These two facts together verify the first three norm equations. For the final one, note that
\begin{equation*}
3\left(- \frac{1+\eta}{2}\right)^2 + 3\left(- \frac{1+\eta}{2}\right)+1 = \frac{3+3\eta^2 +6\eta}{4} - \frac{3+3\eta}{2} + 1 = \frac{1 + 3\eta^2}{4} = v.
\end{equation*}
This implies that the last norm equation holds as well, finishing the proof.
\end{proof}
For the proof of Theorem~\ref{thm:main} we need one further ingredient.
\begin{theorem}[Chebotar\"ev's density theorem, \cite{Stevenhagen96}]
\label{thm:chebotarev}
Let $f$ be a polynomial with integer coefficients and with leading coefficient $1$. Assume that the discriminant $\Delta(f)$ of f does not vanish. Let $C$ be a conjugacy class of the Galois group $G$ of f. Then the set of primes $p$ not dividing $\Delta(f)$ for which the Frobenius substitution $\sigma_p$ belongs to $C$ has a density, and this density equals $\frac{|C|}{|G|}$.
\end{theorem}
In what follows we shall use the terminology from \cite{Stevenhagen96}. In particular, if $K$ is a field of characteristic $0$ and $p$ a prime number, a \textit{place} of $K$ over $p$ is a map $\psi:K \rightarrow \overline{\mathbb{F}}_p \cup \{\infty\}$ with the following properties:
\begin{itemize}
\item[(i)] $\psi^{-1}{\overline{\mathbb{F}}_p}$ is a subring of $K$, and $\psi : \psi^{-1}{\overline{\mathbb{F}}_p} \rightarrow \overline{\mathbb{F}}_p$ is a ring homomorphism;
\item[(ii)] $\psi(x) = \infty$ if and only if $\psi(x^{-1}) = 0$, for any non-zero $x \in K$.
\end{itemize}
See \cite{MilneAnt}, Chapter $8$, for a standard treatment of Chebotar\"ev's theorem.
\begin{proof}[Proof of Theorem~\ref{thm:main}]
The polynomial $f(x) := (x^3-2)(x^3 - 3)$ has discriminant $\Delta(f) = 26244$ and splitting field $F$. Applying Theorem~\ref{thm:chebotarev} to the conjugacy class $C_\sigma$ of $\sigma$ in the Galois group $G$ of $f$ gives an infinite sequence $\mathcal{S}$ of primes of density $\frac{|C_\sigma|}{|G|} = \frac{1}{9}$.

Let $p$ be any prime in $\mathcal{S}$ not dividing the discriminant $\Delta(g) = -248832$ of $g$, and $\psi$ a place of $F$ over $p$. By assumption, $\sigma_p \in C_\sigma = \{\sigma, \sigma^2\}$. Thus the Frobenius automorphism of $\mathbb{F}_p$ fixes the elements of $\psi(E)$, and permutes cyclically the roots of $x^3 - 2$. 

Hence $\psi$ maps any root of $g$ to an element of $\mathbb{F}_p \cup \{\infty\}$. However, as $x^3g(\frac{1}{x})$ has constant term $1$, none of the roots of $g$ map to $\infty$. As $p \not| \Delta(g)$, all the roots of $g$ are distinct elements of $\mathbb{F}_p$.

For the same reason there is a cube root of unity present in $\mathbb{F}_p$.

As the roots of $x^3 - 2$ are permuted cyclically, $x^3 - 2$ stays irreducible over $\mathbb{F}_p$. Let us adjoin the root $\psi(\theta_1)$ of $x^3 - 2$ to form $\mathbb{F}_{p^3}$. Then the sets $\psi(A)$ and $\psi(B)$ are well-defined in $\mathbb{F}_{p^3} / \mathbb{F}_p$. Furthermore $\psi(N_{F/E}(x)) = N_{\mathbb{F}_{p^3} / \mathbb{F}_p}(\psi(x))$ holds for any $x = \alpha+\beta$, with $(\alpha, a) \in A, (\beta, b) \in B$. Then Lemma~\ref{lem:K46} shows that $\psi(A)$ and $\psi(B)$ form a $K_{4, 6}$ in $\mathcal{P}(p, 4)$.
\end{proof} 

\section{\normalsize The construction of $K_{t-1, m}$}

In this section we prove Theorem~\ref{thm:sec}. Let $t \geq 4$ and $m \geq 1$ be arbitrary. Let $\theta_1, \theta_2, \ldots, \theta_m$ be the roots of $x^m - 2$. This time there is no particular reason for choosing this polynomial; any $m$ conjugates will do. Let $L$ be the splitting field of $(x^m-2)(x^{t-2} - 1)$.

Define the polynomials $f_i(x) = x^{t-1} - x + \theta_i, 1 \leq i \leq m$. The derivative $f_1'(x) = (t-1)x^{t-2} - 1$ has $t-2$ distinct roots $x_1, \ldots, x_{t-2}$. Then
\begin{equation*}
f_1(x_i) = \frac{x_i}{t-1} - x_i + \theta_1 = -(1 - \frac{1}{t-1})x_i + \theta_1.
\end{equation*}
Hence $f_1$ takes distinct values at $x_i, 1 \leq i \leq t-2$. We now use the following theorem of Hilbert.
\begin{theorem}[\cite{serre92}, Theorem 4.4.5]
\label{thm:hilbert1}
Let $K$ be any field of characteristic $0$, and $g \in K[x]$ of degree $n$. Suppose $g'(x)$ has distinct zeroes, and $g$ takes distinct values at the zeroes of $g'$. Then $g(x) - T$ has Galois group $S_n$ over $K(T)$.
\end{theorem}
Therefore $f_1(x) - T$ is irreducible and with Galois group $S_{t-1}$ over $L(T)$. By Hilbert Irreducibility (\cite{serre92}, Theorem~3.4.1 and Proposition~3.3.5), we can find an integer $r$ such that $f_1(x) - r$ stays irreducible over $L$ and still has Galois group $S_{t-1}$.

Let $F$ be the splitting field of $f_1(x) - r$ over $L$. Let $\sigma$ be the Galois automorphism of $F$ corresponding to the element $(1 2 \ldots t-1)$ in the Galois group of $F/L$. Then $\sigma$ generates a subgroup $H = \{1, \sigma, \ldots, \sigma^{t-2}\}$. Let $E := F^H$ be the fixed field of this subgroup. 

\begin{lemma}
\label{lem:irreduc}
Let $\theta_i \in E$ be any conjugate of $\theta_1$. Then $\sigma$ permutes the roots of the polynomial $f_i(x) - r = x^{t-1} - x + \theta_i - r$ cyclically, for any $1 \leq i \leq m$.
\end{lemma}
\begin{proof}
Take a $\mathbb{Q}$-automorphism of $L$ mapping $\theta_1$ to $\theta_i$, and extend it to an automorphism $\tau$ of $F$ mapping $E$ to itself.

As $f_1(x) - r$ splits over $F$, so does $f_i(x) - r = \tau(f_1(x) - r)$. Hence if $\sigma$ does not permute the roots of $f_i(x) - r$ cyclically, then $f_i(x) - r$ factors as $g_i(x)h_i(x)$ over $E$. It follows that $f_1(x) - r = \tau^{-1}(f_i(x) - r) = \tau^{-1}(g_i(x))\tau^{-1}(h_i(x))$, and so $f_1(x) - r$ also factors over $E$, a contradiction with the fact that $f_1(x) - r$ is irreducible over $E$.
\end{proof}

Now apply Theorem~\ref{thm:chebotarev} to the minimal polynomial of a primitive element in $F / \mathbb{Q}$ and the conjugacy class of $\sigma$ in $\gal(F/\mathbb{Q})$, to obtain a sequence of primes of positive density. Let $p$ be any prime in this sequence, and $\psi$ a place of $F$ over $p$. Then $\psi$ maps all the roots $\theta_i$ to elements of $\mathbb{F}_p$. By design, $\psi(f_1(x) - r)$ is an irreducible polynomial of degree $t-1$ over $\mathbb{F}_p$: adjoin some root to obtain $\mathbb{F}_{p^{t-1}}$. By Lemma~\ref{lem:irreduc}, all the polynomials $\psi(f_i(x) - r)$ are irreducible over $\mathbb{F}_p$, and split in $\mathbb{F}_{p^{t-1}}$. Let $\alpha_i$ be a root of the polynomial $\psi(f_i(x) - r), 1 \leq i \leq m$. Further let $\zeta$ be a $(t-2)$-root of unity, which must belong to $\mathbb{F}_p$ by construction.

Now define $A := \{(\zeta^k, 1) : 1 \leq k \leq t-2\} \cup \{(0, 1)\}$.

Also let $B := \{(-\alpha_i, \psi(\theta_i - r)) : 1 \leq i \leq m\}$.

Then $|A| = t-1, |B| = m$ and $A \cup B$ are vertices of $\mathcal{P}(p, t)$. We claim that they form a $K_{t-1, m}$ in $\mathcal{P}(p, t)$. To see this, note that for any $1 \leq k \leq t-2$ and $1 \leq i \leq m$,
\begin{equation*}
N_{\mathbb{F}_{p^{t-1}} / \mathbb{F}_p}(\zeta^k + (-\alpha_i)) = \psi(f_i-r)(\zeta^k) = \psi(\theta_i - r).
\end{equation*}
Furthermore,
\begin{equation*}
N_{\mathbb{F}_{p^{t-1}} / \mathbb{F}_p}(0 + (-\alpha_i)) = \psi(f_i-r)(0) = \psi(\theta_i - r)
\end{equation*}
as well. This shows that $A \cup B$ indeed defines a $K_{t-1, m}$, as desired.

\medskip
\textit{Acknowledgements.}
The question answered in this paper was raised during the Discrete Mathematics III seminar held at Freie Universit\"at Berlin in the winter of 2015. I would like to thank all the participants of the seminar for their motivating talks. 

\bibliographystyle{abbrv}
\bibliography{bibl}
\end{document}